\documentclass[12pt]{amsart}
\usepackage{amscd,amsmath,amsthm,amssymb,graphics}
\usepackage{pstcol,pst-plot,pst-3d}%\usepackage[T1]{fontenc}
\usepackage{lmodern,pst-node}%\usepackage{geometric}
\usepackage{multicol}
\usepackage{epic,eepic}
\usepackage{amsfonts,amssymb,amscd,amsmath,enumerate,verbatim}
\psset{unit=0.7cm,linewidth=0.8pt,arrowsize=2.5pt 4}
\newpsstyle{fatline}{linewidth=1.5pt}
\newpsstyle{fyp}{fillstyle=solid,fillcolor=verylight}
\definecolor{verylight}{gray}{0.97}
\definecolor{light}{gray}{0.9}
\definecolor{medium}{gray}{0.85}
\definecolor{dark}{gray}{0.6}

\def\frk{\frak}               % font for "Fraktur"
\newcommand{\fm}{\frak{m}}
\newcommand{\fp}{\frak{p}}

\def\mm{{\frk m}}

\def\Phi{{\frk n}}
\def\Phi{{\frk N}}
\def\opn#1#2{\def#1{\operatorname{#2}}} % to make operators
\opn\chara{char} \opn\length{\ell} \opn\pd{pd} \opn\rk{rk}
\opn\projdim{proj\,dim} \opn\injdim{inj\,dim} \opn\rank{rank}
\opn\depth{depth} \opn\grade{grade} \opn\height{height}
\opn\embdim{emb\,dim} \opn\codim{codim}

\opn\Tr{Tr} \opn\bigrank{big\,rank}
\opn\superheight{superheight}\opn\lcm{lcm}
\opn\trdeg{tr\,deg}%\emph{
\opn\reg{reg} \opn\lreg{lreg} \opn\ini{in} \opn\lpd{lpd}
\opn\size{size}\opn\bigsize{bigsize}
\opn\cosize{cosize}\opn\bigcosize{bigcosize}
\opn\sdepth{sdepth}\opn\sreg{sreg}
\opn\link{link}\opn\fdepth{fdepth}
\opn\div{div} \opn\Div{Div} \opn\cl{cl} \opn\Cl{Cl}
\opn\Spec{Spec} \opn\Supp{Supp} \opn\supp{supp} \opn\Sing{Sing}
\opn\Ass{Ass} \opn\Min{Min}\opn\Mon{Mon} \opn\dstab{dstab} \opn\astab{astab}
\opn\Syz{Syz}
\opn\Ann{Ann} \opn\Rad{Rad} \opn\Soc{Soc}
\opn\Im{Im} \opn\Ker{Ker} \opn\Coker{Coker} \opn\Am{Am}
\opn\Hom{Hom} \opn\Tor{Tor} \opn\Ext{Ext} \opn\End{End}
\opn\Aut{Aut} \opn\id{id}

\opn\nat{nat}
\opn\pff{pf}%   \pf exists already
\opn\Pf{Pf} \opn\GL{GL} \opn\SL{SL} \opn\mod{mod} \opn\ord{ord}
\opn\Gin{Gin} \opn\Hilb{Hilb}\opn\sort{sort}
\opn\initial{init}
\opn\ende{end}
\opn\height{height}
\opn\type{type}
\opn\aff{aff} \opn\con{conv} \opn\relint{relint} \opn\st{st}
\opn\lk{lk} \opn\cn{cn} \opn\core{core} \opn\vol{vol}
\opn\link{link} \opn\star{star}\opn\lex{lex}
\opn\gr{gr}

\def\pot#1#2{#1[\kern-0.28ex[#2]\kern-0.28ex]}

\opn\dirlim{\underrightarrow{\lim}}
\opn\inivlim{\underleftarrow{\lim}}
\let\union=\cup
\let\sect=\cap

\let\to=\rightarrow

\def\Implies{\ifmmode\Longrightarrow \else
        \unskip${}\Longrightarrow{}$\ignorespaces\fi}
\def\implies{\ifmmode\Rightarrow \else
        \unskip${}\Rightarrow{}$\ignorespaces\fi}
\def\iff{\ifmmode\Longleftrightarrow \else
        \unskip${}\Longleftrightarrow{}$\ignorespaces\fi}

\let\:=\colon
\newtheorem{Theorem}{Theorem}[section]
 
 \newtheorem{Corollary}[Theorem]{Corollary}
 
 \newtheorem{Remark}[Theorem]{Remark}
 
 \newtheorem{Example}[Theorem]{Example}

\let\epsilon\varepsilon
\let\kappa=\varkappa
\def\qed{\ifhmode\textqed\fi
      \ifmmode\ifinner\quad\qedsymbol\else\dispqed\fi\fi}
\def\textqed{\unskip\nobreak\penalty50
       \hskip2em\hbox{}\nobreak\hfil\qedsymbol
       \parfillskip=0pt \finalhyphendemerits=0}
\def\dispqed{\rlap{\qquad\qedsymbol}}

\opn\dis{dis}
\def\pnt{{\raise0.5mm\hbox{\large\bf.}}}

\opn\Lex{Lex}

\begin{document}
 \title{Stability properties of powers of ideals in regular local rings of small dimension}

 \author {J\"urgen Herzog and Amir Mafi}

\address{J. Herzog, Fachbereich Mathematik
Universit\"at Duisburg-Essen, Campus Essen, 45117 Essen, Germany}
\email{juergen.herzog@uni-essen.de}

\address{A. Mafi, Department of Mathematics, University of Kurdistan, P.O. Box: 416, Sanandaj,
Iran.}
\email{a\_mafi@ipm.ir}



\begin{abstract}
 Let $(R,\mm)$ be a regular local ring  or a polynomial ring  over a field,  and let $I$ be an ideal of $R$ which we assume to be graded if $R$ is a
 polynomial ring. Let $\astab(I)$ resp.\ $\overline{\astab}(I)$ be the smallest integer $n$ for which
  $\Ass(I^n)$ resp.\ $\Ass(\overline{I^n})$ stabilize, and  $\dstab(I)$ be  the smallest integer $n$  for which $\depth(I^n)$ stabilizes. Here $\overline{I^n}$ denotes the integral closure of $I^n$.

 We show that  $\astab(I)=\overline{\astab}(I)=\dstab(I)$  if $\dim R\leq 2$, while already in dimension $3$, $\astab(I)$ and $\overline{\astab}(I)$ may differ by any amount. Moreover, we show that if  $\dim R=4$, then there exist ideals $I$ and $J$ such that for any positive integer $c$ one has $\astab(I)-\dstab(I)\geq c$ and   $\dstab(J)-\astab(J)\geq c$.
\end{abstract}

\maketitle

\section*{Introduction}
Let $(R,\mm)$ be a commutative Noetherian ring and $I$ be an ideal of $R$. Brodmann \cite{B1} proved that the set of associated prime ideals $\Ass(I^k)$
stabilizes. In other words, there exists an integer $k_0$ such that $\Ass(I^k)=\Ass(I^{k_0})$ for all $k\geq k_0$. The smallest such integer $k_0$ is
called the {\em index of Ass-stability}  of $I$,  and denoted by $\astab(I)$. Moreover,   $\Ass(I^{k_0})$ is called the {\em stable set of associated
prime ideals} of $I$. It is denoted by $\Ass^{\infty}(I)$. For the integral closures $\overline{I^k}$ of the powers of $I$,  McAdam and Eakin \cite{Me}
showed that  $\Ass({\overline{I^k}})$ stabilizes as well. We denote the index of stability for the integral closures of the powers of $I$ by
$\overline{\astab}(I)$,  and denote its   stable set of associated prime ideals by $\overline{\Ass}^{\infty}(I)$.

Brodmann \cite{B} also showed that  $\depth R/I^k$ stabilizes. The  smallest power of $I$  for which  depth stabilizes is  denoted by $\dstab(I)$. This
stable depth is called the {\em limit depth} of $I$, and is denoted  by $\lim_{k\to\infty}\depth R/I^k$. These  indices  of stability have been  studied
and  compared to some extend in \cite{Hq} and \cite{Hrv}. The purpose of this work is to compare once  again these  stability indices. The main result is
that if $(R,\mm)$ is a regular local ring with $\dim R\leq 2$, then  all 3 stability indices are equal, but  if $\dim R=3$, then
we still have $\astab(I)=\dstab(I)$, while  $\astab(I)$ and $\overline{\astab}(I)$ may differ by any amount. On the other hand, if $\dim R\geq 4$, we will
show by examples that in general a comparison between these stability indices is no longer possible. In other words, any inequality between these
invariants may occur.

Quite often, but not always,  $\depth(R/I^k)$ is a non-increasing function on $n$. In the last section  we prove that if $(R,\mm)$ is a $3$-dimensional
regular local ring and $I$ satisfies $I^{k+1}:I=I^k$ for all $k$, then $\depth R/I^k$ is non-increasing.
For any unexplained notion or terminology, we refer the reader to \cite{Bh}.

Several explicit example were  performed with help of the computer algebra systems CoCoA \cite{Ab} and Macaulay2 \cite{Gs},  as well as with the program in
\cite{Bhr} which allows to compute $\Ass^\infty(I)$ of a monomial ideal $I$.

\section{The  case $\dim R\leq 3$}

In this section we study the behavior of the stability indices for regular rings of dimension~$\leq 3$. In the proofs we will use the following elementary
and well-known fact: let $I\subset R$ be an ideal of height $1$ in the regular local ring $R$. Then there exists $f\in R$ such that $I=fJ$ where either
$J=R$ or otherwise $\height J>1$. Indeed, let $I=(f_1,\ldots,f_m)$. Since $R$ is factorial,  the greatest common   divisor  of $f_1,\ldots,f_m$ exists.
Let $f=\gcd(f_1,\ldots,f_m)$, and $g_i=f_i/f$ for $i=1,\ldots,m$. Then $I=fJ$, where $J=(g_1,\ldots,g_m)$. Suppose $\height J=1$, then there exist a prime
ideal $P$ of height 1 with $J\subset P$. Since $R$ is regular, $P$ is a principal ideal, say $P=(g)$. It follows then that $g$ divides all $g_i$, but
$\gcd(g_1,\ldots,g_m)=1$, a contradiction.

%For an ideal $I$ we denote by $\ell(I)$ the analytic spread if $I$,  that is,   the Krull dimension of the fiber $R[It]/\mm  R[IT]$ of the Rees ring
%$$R[It]=R\dirsum It\dirsum I^2t^2\dirsum\cdots $.

\medskip
We first observe
\begin{Remark}
\label{r}
Let $(R,\mm)$ be a regular local ring with $\dim R\leq 2$ and let  $I$ be an ideal of $R$. Then
\[
\astab(I)=\overline{\astab}(I)=\dstab(I)=1.
\]
\end{Remark}

\begin{proof}
If $\dim R\leq 1$, then either $R$ is a field or a principal ideal domain, and the statement is trivial.  Now suppose that $\dim R=2$ and $I\neq 0$. If
$\height I=2$, then $\mm$ belongs to $\Ass(I^k)$ and $\Ass(\overline{I^k})$  for all $k$, and the assertion is trivial.  Hence, we may  assume that
$\height(I)=1$. Then $I=fJ$ with $J=R$ or $\height J=2$. In the first case $I$ is a principal ideal, and the assertion is trivial. In the second case,
$I^k=f^kJ^k$ for all $k$, and $J^k$ is $\mm$-primary. Thus there exists $g\not \in J^k$ with $g\mm\in J^k$. Then $gf^k\not\in f^kJ^k$ and   $gf^k\mm\in
f^kJ^k$. This shows that in the second case $\mm\in \Ass(I^k)$ for $k$, so that $ \astab(I)=\dstab(I)=1$.

Finally observe that in the second case,  $\overline{I^k}=f^k\overline{J^k}$ for all $k$. This shows that $\mm\in \Ass(\overline{I^k})$ for all $k$, so
that also  in this case  $\overline{\astab}(I)=1$.
\end{proof}

\begin{Theorem}
\label{leq3}
Let $(R,\mm)$ be a regular local ring with $\dim R\leq 3$ and $I$ be an ideal of $R$. Then  $\astab(I)=\dstab(I)$.
\end{Theorem}

\begin{proof}
 By Remark~\ref{r}, we may  assume that $\dim R=3$. If $\height I\geq 2$, then $\Ass(I^k)\subseteq\Min(I)\cup\{\mm\}$ for all $k$. This implies at once that
 $\astab(I)=\dstab(I)$. Now suppose that $\height I=1$.  If $I$ is a principal ideal, then the assertion is again trivial.  Otherwise,  $I=fJ$ with
 $\height J\geq 2$.  Since $I^k$ is isomorphic to $J^k$ as an $R$-module, it follows that $\projdim  I^k=\projdim  J^k$ for all $k$. This implies that
 $\projdim  R/I^k=\projdim R/ J^k$ for all $k$, and consequently $\depth R/I^k=\depth  R/ J^k$, by the Auslander--Buchsbaum formula.  Thus,
 $\dstab(I)=\dstab(J)$.

We claim that $\astab(I)=\astab(J)$. Since we have already seen that $\astab(J)=\dstab(J)$ if $\height J\geq 2$,  the claim then implies that
$\astab(I)=\dstab(I)$, as desired.

The claim follows once we have shown that $\Ass(I^k)=\Ass(f^kJ^k)=\Min(f)\union \Ass(J^k)$. For that we only need to prove the second equation. So let
$P\in \Spec R$ with $f^kJ^k\subset P$. Then $P\in \Ass(f^kJ^k)$ if and only if $\depth R_P/f^kJ^kR_P=0$. If $J\not\subset P$, then $f^kJ^kR_P=f^kR_P$, and
hence $\depth R_P/f^kJ^kR_P=0$ if and only if  $\depth R_P/f^kR_P=0$, and this is the case if and only if $P\in \Min(f)$. If $J\subset P$, then the
$R_P$-modules $f^kJ^kR_P$ and $J^kR_P$ are isomorphic, so that with the arguments as above $\depth R_P/f^kJ^kR_P=\depth R_P/J^kR_P$, which shows that in
this case $P\in \Ass(f^kJ^k)$ if and only if $P\in \Ass(J^k)$. This completes the proof.
\end{proof}

The  statements shown so far  and its proofs made for ideals in a regular local ring  are valid as well for any graded ideal in a polynomial ring.

We now turn to some explicit examples. In \cite[Proposition 1.5]{Hmst} Hibi et al show that for  any integer $t\geq 2$ the ideal
$I=(x^t,xy^{t-2}z,y^{t-1}z)\subset K[x,y,z]$ satisfies $\dstab(I)=t$. Since by Theorem~\ref{leq3}, $\astab(I)=\dstab(I)$, this example shows that in a
$3$-dimensional graded or local ring (we may pass to  $K[|x,y,z|]$) both the index of depth stability as well as the index of Ass stability may be any
given number.

\medskip
The following example shows that already for an ideal $I$ in a $3$-dimensional polynomial ring the invariants $\astab(I)$ and $\overline{\astab}(I)$ may
differ.

\begin{Example}
Let  $R=K[x,y,z]$  be a polynomial ring over a field $K$ and let $I=((xy)^2,(xz)^2,(yz)^2)\subset R$. Then
$\astab(I)=2$ and $\overline{\astab}(I)=1$.
\begin{proof}
We first claim that $I^n:(xy)^2=I^{n-1}+z^{2n}(x^2,y^2)^{n-2}$. Indeed,  let $J=((xz
)^2,(yz)^2)$. Then $I^n=J^n+(xy)^2I^{n-1}$,  and hence $I^n:(xy)^2=J^n:(xy)^2+I^{n-1}$. Since $J^n:(xy)^2=z^{2n}(x^2,y^2)^n:(xy)^2=z^{2n}(x^2,y^2)^{n-2}$,
the assertion follows.

By symmetry,  we also have $I^n:(xz)^2=I^{n-1}+y^{2n}(x^2,z^2)^{n-2}$ and $I^n:(yz)^2=I^{n-1}+x^{2n}(y^2,z^2)^{n-2}$. Thus, for all $n\geq 1$ we obtain
\begin{eqnarray*}
I^n:I&=&(I^n:(xy)^2)\sect  (I^n:(xz)^2)\sect(I^n:(yz)^2)\\
&=&(I^{n-1}+z^{2n}(x^2,y^2)^{n-2})\sect (I^{n-1}+y^{2n}(x^2,z^2)^{n-2})\sect (I^{n-1}+x^{2n}(y^2,z^2)^{n-2})\\
&=& I^{n-1}+(z^{2n}(x^2,y^2)^{n-2})\sect (y^{2n}(x^2,z^2)^{n-2})\sect (x^{2n}(y^2,z^2)^{n-2})=I^{n-1}.
\end{eqnarray*}
In other words,  $I$ satisfies strong persistence in the sense of \cite{Hq}. In particular, $\Ass(I^n)\subset \Ass(I^{n+1})$ for all $n\geq 1$. Now since
$\Ass(I)=\{(x,y),(x,z),(y,z)\}$ and $\Ass(I^2)=\{(x,y),(x,z),(y,z),(x,y,z)\}$, we deduce from this that $\astab(I)=\dstab(I)=2$.

With Macaulay2  it can be checked that  $\overline{I}=\{(xy)^2,(xz)^2,(yz)^2,xyz^2,xy^2z,x^2yz\}$ and that
$\Ass(\overline{I})=\{(x,y),(x,z),(y,z),(x,y,z)\}$. By \cite[Corollary 11.28]{M}, one has $\Ass(\overline{I})\subset \Ass(\overline{I^2})\subset \cdots
\subset \overline{\Ass}^\infty(I)$. Since $\Ass(\overline{I^n})$ is a subset of the monomial prime ideals containing $I$, and since this set is
$\{(x,y),(x,z),(y,z),(x,y,z)\}$, we see that $\Ass(\overline{I})=\Ass(\overline{I^n})$ for all $n$. Hence, $\overline{\astab}(I)=1$.
\end{proof}
\end{Example}

The next result says that the difference $\astab(I)-\overline{\astab}(I)$ may  indeed  be as big as we want.

\begin{Theorem}
Let $R=k[x,y,z]$ be the polynomial ring over a field $K$, $c$ a positive integer and $I=(x^{2c+2},xy^{2c}z,y^{2c+2}z)$.
Then $\astab(I)=c+2$ and  $\overline{\astab}(I)=2$.
\end{Theorem}

\begin{proof} Note that $I=(x^{2c+2},z)\sect (x,y^{2c+2})\sect (y^{2c},x^{2c+2})$, from which it follows that $\dim R/I=\depth R/I=1.$

 In the next step we prove that $I^n:I=I^{n-1}$ for all $n$. Then \cite[Theorem 1.3]{Hq} implies that  $\Ass(I^n)\subseteq\Ass(I^{n+1})$ for all $n$. In
 particular, if $\depth(R/I^k)=0$ for some $k$, then $\depth(R/I^r)=0$ for all $r\geq 0$. Since $\depth R/I^k\leq 1$ for all $k$, it then follows that
 $\depth(R/I^k)\geq \depth(R/I^{k+1})$ for all $k$.

\medskip
In order to show that $I^n:I=I^{n-1}$,  observe that $$I^n:x^{2c+2}=I^{n-1}+((y^{2c}z)^n(x,y^2)^n:x^{2c+2})=I^{n-1}+(y^{2c}z)^n(x,y^2)^{n-2(c+1)},$$ and
that
\begin{eqnarray*}
I^n:xy^{2c}z&=&I^{n-1}+((x^{2c+2},y^{2c+2}z)^n:xy^{2c}z)\\
&\subseteq& I^{n-1}+(((x^{2c+2},y^{2c+2}z)^n:y^{2c+2}z):x^{2c+2})\\
&=&I^{n-1}+(x^{2c+2},y^{2c+2}z)^{n-2}.
\end{eqnarray*}
Similarly we have
\begin{eqnarray*}
I^n:y^{2c+2}z&=&I^{n-1}+(x^n(x^{2c+1},y^{2c}z)^n:y^{2c+2}z)\\
&\subseteq &I^{n-1}+(x^n(x^{2c+1},y^{2c}z)^n:y^{4c}z^2)\\
&=&I^{n-1}+x^n(x^{2c+1},y^{2c}z)^{n-2}.
\end{eqnarray*}
Now since
\begin{eqnarray*}
I^{n-1}&\subseteq &(I^n:I)\\
&\subseteq & I^{n-1}+(y^{2c}z)^n(x,y^2)^{n-2(c+1)}\cap(x^{2c+2},y^{2c+2}z)^{n-2}\cap x^n(x^{2n+1},y^{2c}z)^{n-2}\\
&\subseteq& I^{n-1}+I^n=I^{n-1},
\end{eqnarray*}
it follows $I^n:I=I^{n-1}$ for all $n$, as desired.

\medskip
Next we claim that $I^n:x^{2c+2}=I^{n-1}$ for all $n\leq{c+1}$.

If $n=1$, there is nothing to prove. Let $1<n\leq{c+1}$.
By a calculation as before we see that
\begin{eqnarray*}
I^n:x^{2c+2}&=&I^{n-1}+((y^{2c}z)^{n}(x,y^2)^{n}:x^{2c+2})=I^{n-1}+(y^{2c}z)^{n}\\
&=&I^{n-1}+(y^{(2c+2)(n-1)+2c+2-2n}z^n)=I^{n-1}+(y^{2c+2}z)^{n-1}y^{2c+2-2n}z\\
&=&I^{n-1}.
\end{eqnarray*}

We proceed by induction $n$ to show that $\depth S/I^n=1$ for $n\leq c+1$. We observed already that $\depth S/I=1$, Now let $1<n\leq c+1$.  Then, since
$I^n:x^{2c+2}=I^{n-1}$, we obtain the exact sequence
$$0\longrightarrow R/I^{n-1}\overset{x^{2c+2}}\longrightarrow R/I^n\longrightarrow R/(I^n,x^{2c+2})\longrightarrow 0.$$
Since by induction hypothesis $\depth R/I^{n-1}=1$,  it follows that
$$\depth R/I^n\geq\min\{\depth R/I^{n-1},\depth R/(I^n,x^{2c+2})\}=\min\{1,\depth R/(I^n,x^{2c+2})\}.$$
Note that $(I^n,x^{2c+2})=((y^{2c}z(x,y^2)^n, x^{2c+2})$, which implies that $\depth  R/(I^n,x^{2c+2})=1$. Thus we have $\depth R/I^n\geq 1$.
On the other hand,  we have seen before that    $\depth R/I^n\leq\depth R/I=1$, and so   $\depth R/I^n=1$ for all $n\leq{c+1}$.

\medskip
In the next step we show that $\depth R/I^{c+2}=0$, which then implies that $\depth R/I^{n}=0$ for all $n\geq c+2$. In  particular,  it will follows that
$\astab(I)=c+2$.

In order to prove that  $\depth R/I^{c+2}=0$, we show that $x^{2c+2}y^{(c+1)(2c+2)-1}z^{c+1}\in (I^{c+2}:\fm)\setminus I^{c+2}$.

Indeed, let $u=x^{2c+2}y^{(c+1)(2c+2)-1}z^{c+1}$. Then
\[
ux=x^{2c+2}(xy^{2c}z)(y^{2c+2}z)^cy^{2c}, \quad  uy=x^{2c+2}(y^{2c+2}z)^{c+1}
\]
and
\[  uz= (xy^{2c}z)^{c+1}(xy^{2c}z)(yx^c).
\]
This shows that $u\in (I^{c+2}:\fm).$

Assume that $x^{2c+2}y^{(c+1)(2c+2)-1}z^{c+1}\in I^{c+2}$. Then $$y^{(c+1)(2c+2)-1}z^{c+1}\in (I^{c+2}:x^{2c+2})=I^{c+1}+(y^{2c}z)^{c+2},$$
and so $y^{(c+1)(2c+2)-1}z^{c+1}\in I^{c+1}$. Since $I^{c+1}=(x^{2c+2},y^{2c}z(x,y^2))^{c+1}$, expansion of this power implies that
$y^{(c+1)(2c+2)-1}\in\sum_{i=0}^{c+1}(x^{2c+2})^i(y^{2c}(x,y^2))^{c+1-i}$. It follow that  $y^{(c+1)(2c+2)-1}\in(y^{2c}(x,y^2))^{c+1}$, which is a
contradiction.

\medskip
Now we compute $\overline{\astab}(I)$, and first prove that
\[
\overline{I}=(I,(x^3y^{2c-1}z,x^4y^{2c-2}z,\ldots,x^{2c+1}yz)).
\]

Let $J=(I,(x^3y^{2c-1}z,x^4y^{2c-2}z,\ldots,x^{2c+1}yz))$. For all integers $i$ with $3\leq i\leq{2c+1}$, we have
\begin{eqnarray*}
(x^iy^{2c-i+2}z)^{2c}&=&x^{2ic}y^{2c(2c-i+2)}z^{2c}=x^{2c(i-1)+i-2}x^{2c-i+2}y^{2c(2c-i+2)}z^{2c-i+2}z^{i-2}\\
&=&
x^{2c(i-1)+i-2}(xy^{2c}z)^{2c-i+2}z^{i-2}\\
&=&(x^{2c+2})^{i-2}(xy^{2c}z)^{2c-i+2}z^{i-2}x^{2c+2-i}\in I^{2c}.
\end{eqnarray*}
Thus $J\subseteq\overline{I}$. We have $\Ass(\overline{I}/J)\subseteq\Ass(J)$. The  primary decomposition of $J$ shows that  $\Ass(J)=\{(x,z),(x,y)\}$. Let
$P=(x,z)$. Then $(\overline{I})_P=\overline{(I_P)}=\overline{(x^{2c+2}, z)_P}=(x^{2c+2},z)_P$.  The last equality follows by \cite[Proposition 1.3.5]{Hs},
and so $(\overline{I}/J)_{P}=0$. Hence
$P\notin\Ass(\overline{I}/J)$. Now let $P=(x,y)$. Then
\begin{eqnarray*}
(\overline{I})_P&=&\overline{(x^{2c+2},xy^{2c},y^{2c+2})_P}\subset \overline{((x,y)^{2c+2},xy^{2c})_P}=((x,y)^{2c+2},xy^{2c})_{P}=J_{P}.
\end{eqnarray*}
The second  equality follows by \cite[Exercise 1.19]{Hs}.  Thus we have $(\overline{I}/J)_P=0$. This shows that $\Ass(\overline{I}/J)=\emptyset$, and hence
$\overline{I}=J$, as desired. In particular we see that $$\Ass(\overline{I})=\{(x,z),(x,y)\}.$$

Since $\Ass(\overline{I})\subseteq \Ass(\overline{I^k})$ for all $k$, it follows that $\{(x,z),(x,y)\}\subset \Ass(I^k)$ for all $k$. Suppose that
$(y,z)\in \Ass(\overline{I^k})$ for some $k$.  Then $(y,z)$ is a minimal prime ideal of $I$. However, this is not the case, as can be seen from the primary
decomposition of $I$.

Next we show that $\mm=(x,y,z)$ belongs to $\Ass(\overline{I^2})$.  Then it follows that
\[
\Ass(\overline{I^k})=\{(x,z),(x,y), (x,y,z)\} \quad \text{for all}\quad k\geq  2,
\]
thereby showing that $\overline{\astab}(I)=2$.

\medskip
In order to prove that $\mm\in \Ass(\overline{I^2})$, we first show that the ideal
\begin{eqnarray*}
L&=&(I^2,(x^4y^{4c-1}z^2,x^5y^{4c-2}z^2,\ldots,x^{2c+2}y^{2c+1}z^2),\\
&&(x^{2c+5}y^{2c-1}z,x^{2c+6}y^{2c-2}z\ldots,x^{4c+3}yz)).
\end{eqnarray*}
is contained in $\overline{I^2}$.

Since  $$I^2=(x^{4c+4},x^2y^{4c}z^2,y^{4c+4}z^2,x^{2c+3}y^{2c}z,x^{2c+2}y^{2c+2}z,xy^{4c+2}z^2),$$
it follows that for all integers $i$ with  $4\leq i\leq{2c+2}$ the element
\begin{eqnarray*}
(x^iy^{4c-i+3}z^2)^{4c}x^{4ic}y^{4c(4c-i+3)}z^{8c}=x^{2(4c-i+3)}y^{4c(4c-i+3)}z^{2(4c-i+3)}x^{4c(i-2)+2i-6}z^{2i-6}\\
=(x^2y^{4c}z^2)^{4c-i+3}(x^{4c+4})^{i-3}x^{4c-2i+6}z^{2i-6}\hspace{6.7cm}
\end{eqnarray*}
belongs to $(I^2)^{4c}$. Also, for all integers $i$ with $5\leq i\leq{2c-2}$, the element
\begin{eqnarray*}
(x^{2c+i}y^{2c+4-i}z)^{4c}&=&x^{2(2c+4-i)}y^{4c(2c+4-i)}z^{2(2c+4-i)}x^{8c^2+4ic+2i-4c-8}z^{2i-8}\\
&=&(x^2y^{4c}z^2)^{2c+4-i}x^{(4c+4)(2c+i-4)}x^{4c+8-2i}z^{2i-8}\\
&=&(x^2y^{4c}z^2)^{2c+4-i}(x^{4c+4})^{2c+i-4}x^{4c+8-2i}z^{2i-8}
\end{eqnarray*}
belongs to $(I^2)^{4c}$,   This shows $L\subseteq\overline{I^2}$.

By using primary decomposition for the ideal $L$, we see that  $$\Ass(L)=\{(x,z),(x,y),(x,y,z)\}.$$
 On the other hand, by easy calculation, one verifies that
$L:(x^{2c+2}y^{2c+1}z)=\fm$. Finally we show that  $x^{2c+2}y^{2c+1}z\notin \overline{I^2}$,  which then implies that $\mm\in \Ass(\overline{I^2})$, as desired.

In order to prove  this we show by induction on $n$ that  $(x^{2c+2}y^{2c+1}z)^n\notin (I^2)^n$ for all $n$.
Indeed,
let $n=1$,  and assume that $x^{2c+2}y^{2c+1}z\in I^2$. Then $y^{2c+1}z\in I^2:x^{2c+2}=I+(y^{2c}z)^2=I$,  and this is
contradiction.

Now let $n>1$. By induction hypothesis we may assume that  $(x^{2c+2}y^{2c+1}z)^{n-1}\notin (I^2)^{n-1}$.
If $(x^{2c+2}y^{2c+1}z)^n\in (I^2)^n$, then $x^{(2c+2)(n-1)}(y^{2c+1}z)^n\in
(I^{2n}:x^{2c+2})=I^{2n-1}+(y^{2c}z)^{2n}(x,y^2)^{2n-2(c+1)}$,  and so
$x^{(2c+2)(n-1)}(y^{2c+1}z)^n\in I^{2n-1}$.

It follows that  $x^{(2c+2)(n-1)}(y^{2c+1}z)^{n-1}\in (I^{2n-1}:y^{2c+1}z)$. Since
\begin{eqnarray*}
(I^{2n-1}:y^{2c+1}z)&=&
yI^{2n-2}+((x^{2c+2},xy^{2c}z)^{2n-1}:y^{2c+1}z)\\
&=&
yI^{2n-2}+(x^{2n-1}(x^{2c+1},y^{2c}z)^{2n-2}:y)\\
&=&yI^{2n-2}+x^{2n-1}(y^{2c-1}z(x^{2c+1},y^{2c}z)^{2n-3}+(x^{2c+1})^{2n-2}),
\end{eqnarray*}
we see that
$x^{(2c+2)(n-1)}(y^{2c+1}z)^{n-1}\in y(I^2)^{n-1}$,  a contradiction.

Thus $(x^{2c+2}y^{2c+1}z)^n\notin (I^2)^n$ for all
$n$, as desired.
\end{proof}

The theorem says that for any positive integer $c$ there exists a monomial ideal in $K[x,y,z]$ with $\astab(I)-\overline{\astab}(I)=c$. However we do not
know whether for all ideals in $I\subset K[x,y,z]$ one has  $\overline{\astab}(I)\leq \astab(I)$.

\section{The case $\dim R>3$}

The purpose of this section is to show that for  a polynomial ring $S$ in more than  $3$ variables, for a graded ideal $I\subset S$ the invariants $\astab(I)$  and $\dstab(I)$   may differ
by any amount.

\medskip
We begin with two examples.

\begin{Example}
Let $R=k[x,y,z,u]$ be the polynomial ring over a field $k$ and consider the ideal $I=(xy,yz,zu)$
of $R$. Then $\astab(I)=1$ and $\dstab(I)=2$.
\end{Example}
\begin{proof}
We have $\Ass(I)=\Min(I)$, and since I may be viewed as the edge ideal of a bipartite graph it follows from  \cite[Definition 1.4.5, Corollary
10.3.17]{Hh} that $\Ass(I)=\Ass(I^n)$ for all $n\in\mathbb{N}$. Therefore $\astab(I)=1$. By \cite[Corollary 10.3.18]{Hh},
$\lim_{k\to\infty}\depth R/I^k=1$. Moreover, it can be seen that $\depth R/I=2$ and  $\depth R/I^2=1$. Since $I$ has a linear resolution,
\cite[Theorem 10.2.6]{Hh} implies  that for all $k\geq 1$, $I^k$ has a linear resolution as well. Therefore,  by \cite[Proposition 2.2]{Hrv} we have
$\depth R/I^{k+1}\leq\depth R/I^k$ for all $k\in\mathbb{N}$. Hence $\depth R/I^k=1$ for all $k\geq 2$, and so $\dstab(I)=2$.
\end{proof}

\begin{Example}
Let $R=K[x,y,z,u]$ be the polynomial ring in 4 variables  over a field $K$, and let $I=(x^2z,uyz,u^3)$. Then $\astab(I)=2$ and $\dstab(I)=1$.
\end{Example}

\begin{proof}
Set $J=(uyz,u^3)$ and so for all $n\in\mathbb{N}$, it therefore follows that $I^n: x^2z=(J^n+x^2zI^{n-1}): x^2z=I^{n-1}+(J^n:x^2z)=I^{n-1}$. Hence, for all
$n\in\mathbb{N}$, $\Ass(I^{n})\subseteq\Ass(I^{n+1})$. By using Macaulay2 \cite{Gs} and the program \cite{Bhr}, we see that
$\Ass^{\infty}(I)=\Ass(I^2)=\{(x,u),(z,u),(x,y,u),(x,z,u)\}$. Therefore $\astab(I)=2$. As $\Ass(I^{n})\subseteq\Ass(I^{n+1})$ for all $n\in\mathbb{N}$, it follows that 
$\fm=(x,y,z,u)\notin\Ass(I^n)$ and so we have $\depth R/I^n\geq 1$. $y-z\in\fm$ is a nonzerodivisor on $R/I^n$ for all $n\in\mathbb{N}$. Set
$\overline{R}=R/(y-z)$. Thus by \cite[Lemma 4.2.16]{Bh} we have $\overline{(R/I^n)}=\overline{R}/{\overline{I^n}}\cong K[x,z,u]/(x^2z,uz^2,u^3)^n$.
Since $xzu^{3n-1}\in(\overline{I^n}):\overline{\fm}\setminus\overline{I^n}$, it follows $\depth\overline{R}/{\overline{I^n}}=0$ and so $\depth
R/I^n=1$ for all $n\in\mathbb{N}$. Therefore
$\dstab(I)=1$.
\end{proof}

Now we come to the main result of this section

\begin{Theorem}
Let $R=k[x,y,z,u]$ be the polynomial ring over a field $k$. Then for any non-negative integer $c$, there exist two ideals $I$ and $J$ of $R$ such that the
following statements hold:
\begin{itemize}
\item[(i)] $\astab(I)-\dstab(I)\geq c$.
\item[(ii)]  $\dstab(J)-\astab(J)\geq c$.
\end{itemize}
\end{Theorem}

\begin{proof}
We may assume that $c$ is a positive integer. Let  $I=(x^{c+1}z^{c},u^{2c-1}yz,u^{2c+1})$ and $J=(x^{c}y^{c-1},y^{c-1}x^{c-1}z,z^{c}u^{c})$. We claim that
${\astab(I)=\dstab(J)=c+1}$ and ${\astab(J)=\dstab(I)=1}$.

(i) In this case, by using Example 2.2, we can assume that $c\geq 2$.  For all $n\in\mathbb{N}$, we have
$(I^n:x^{c+1}z^c)=(((u^{2c-1}yz,u^{2c+1})^n+x^{c+1}z^cI^{n-1}):x^{c+1}z^c)=I^{n-1}+((u^{2c-1}yz,u^{2c+1})^n:x^{c+1}z^c)$. Since
$((u^{2c-1}yz,u^{2c+1})^n:x^{c+1}z^c)=((u^{2c-1}yz,u^{2c+1})^n:z^c)\subseteq I^{n-1}$, it follows  $(I^n:x^{c+1}z^c)=I^{n-1}$ and so
$\Ass(I^{n})\subseteq\Ass(I^{n+1})$.
By using Macaulay2 \cite{Gs} and \cite{Bhr}, we have $\Ass(I)=\{(x,u),(z,u),(y,z,u),(x,y,u)\}$
 and  $\Ass^{\infty}(I)=\{(x,u),(z,u),(y,z,u),(x,z,u),(x,y,u)\}$. Set ${\fp=(x,z,u)}$. It is easily seen that $I^i:\fp=I^i$ for all ${i\leq c}$ and
 $x^cy^{c+1}z^cu^{(2c+1)c}\in (I_{\fp}^{c+1}:\fp)\setminus I_{\fp}^{c+1}$. Hence $\Ass(I)=\Ass(I^2)=...=\Ass(I^c)$,
 $\Ass(I^{c+1})=\Ass^{\infty}(I)$ and so ${\astab(I)=c+1}$. By the same argument as used in the proof of Example 2.2, that
 $\fm=(x,y,z,u)\notin\Ass(I^n)$ for all $n\in\mathbb{N}$ and so we have $\depth R/I^n\geq 1$ and $x-y-z\in\fm$ is a nonzerodivisor on $R/I^n$ for all $n\in\mathbb{N}$. Therefore
 $\overline{(R/I^n)}=\overline{R}/{\overline{I^n}}\cong K[y,z,u]/((y+z)^{c+1}z^c,u^{2c-1}yz,u^{2c+1})^n$, where $\overline{R}=R/(x-y-z)$. Since
 $z^{2c}u^{(2c+1)n-1}\in(\overline{I^n}):\overline{\fm}\setminus\overline{I^n}$,  it follows $\depth\overline{R}/{\overline{I^n}}=0$ and so $\depth
 R/I^n=1$ for all $n\in\mathbb{N}$. Therefore
$\dstab(I)=1$.

(ii) For all $n\in\mathbb{N}$, we have
$(J^n:z^cu^c)=(((x^{c}y^{c-1},y^{c-1}x^{c-1}z)^n+z^cu^cJ^{n-1}):z^cu^c)=J^{n-1}+((x^{c}y^{c-1},y^{c-1}x^{c-1}z)^n:z^cu^c)=J^{n-1}+((x^{c}y^{c-1},y^{c-1}x^{c-1}z)^n:z^c)$.
Since $((x^{c}y^{c-1},y^{c-1}x^{c-1}z)^n:z^c)\subseteq J^{n-1}$, for all $n\in\mathbb{N}$ we have $(J^n:z^cu^c)=J^{n-1}$. Therefore, for all
$n\in\mathbb{N}$, $\Ass(J^n)\subseteq\Ass(J^{n+1})$. By using \cite{Gs} and \cite{Bhr} we have
$\Ass^{\infty}(J)=\{(x,z),(x,u),(y,z),(y,u)\}=\Min(J)$ and so $\astab(J)=1$. Since $\fm\notin\Ass(J^n)$ for all $n\in\mathbb{N}$, we have $2=\dim
R/J\geq\depth R/J^n\geq 1$  and $x-y\in\fm$ is a nonzerodivisor on $R/J^n$ for all $n\in\mathbb{N}$. Again by the above argument,
$\overline{(R/J^n)}=\overline{R}/{\overline{J^n}}\cong K[x,z,u]/(x^{2c-1},x^{2c-2}z,z^cu^c)^n$, where $\overline{R}=R/(x-y)$. Since
$\overline{J^i}:\overline{\fm}=\overline{J^i}$ for all $i\leq c$ and $x^{(2c-1)n}z^{n-1}u^{c-1}\in\overline{J^n}:\overline{\fm}\setminus\overline{J^n}$ for
all ${n\geq c+1}$.
It therefore follows $\depth R/J=\depth R/J^2=...=\depth R/J^c=2$ and $\depth R/J^n=1$ for all ${n\geq c+1}$. Hence ${\dstab(J)=c+1}$.
\end{proof}

\section{Non-increasing depth functions}

\begin{Theorem}\label{nonincreasing}
Let $(R,\mm)$ be a regular local ring with $\dim R=3$ and $I$ be an ideal of $R$. If  $I^{n+1}:I=I^n$ for all $n\in\mathbb{N}$, then $\depth R/I^n$ is
non-increasing.
\end{Theorem}

\begin{proof}
Suppose $\height(I)\geq 2$. Since $I^{n+1}:I=I^n$ for all $n\in\mathbb{N}$,  it follows that  $\depth R/I^{n+1}\leq\depth R/I^n$.
Now, let $\height(I)=1$. Then there exists an ideal $J$ of $R$ and an element $f\in R$ such that $I=f{J}$ and $\height(J)\geq 2$. As in the proof of
Theorem~\ref{leq3}, $\depth R/I^n=\depth R/J^n$ for all $n\in\mathbb{N}$. Since $I^{n+1}:I=I^{n}$ for all $n\in\mathbb{N}$, we have $J^{n+1}:J=J^n$. Thus
$\depth R/J^{n+1}\leq\depth R/J^n$ and so $\depth R/I^{n+1}\leq\depth R/I^n$. This completes the proof.
\end{proof}

\begin{Corollary}
\begin{itemize}
\item[(i)] Let $(R,\mm)$ be regular local ring  with $\dim R=3$. Then $\depth R/\overline{I^n}$ is non-increasing.
\item[(ii)] Let $R=k[x,y,z]$ be a polynomial ring in $3$ indeterminates over a field $k$. If $I$ is an edge ideal of $R$, then  $\depth R/I^n$ is non-increasing.
\end{itemize}
\end{Corollary}

\begin{Example}{\em
Let $R=k[x,y,z,u]$ be a polynomial ring and $I=(xy^2z,yz^2u,zu^2(x+y+z+u),xu(x+y+z+u)^2,x^2y(x+y+z+u))$
be an ideal of $R$. Then $\depth R/I=\depth R/I^4=0$, $\depth R/I^2=\depth R/I^3=1$. Thus the depth function is neither  non-increasing nor
non-decreasing.}
\end{Example}

In view of Theorem~\ref{nonincreasing} one may ask whether  in a regular local ring (of any dimension), $\depth R/I^n$ is a non-increasing function of $n$,  if $I^{n+1}:I=I^n$ for all $n$.

\subsection*{Acknowledgements}
The second author would like to thank Universit\"at Duisburg-Essen,  especially to the Department
of Mathematics for its hospitality during the preparation of this work. 

%%%%%%%%%%%%%%%%%%%%%%%%%%%%%%%%%%%%%%%%%%%%%%%%%%%%%%%%%%%%%%%%%%%%%%%%%%%%%

\end{document}